\documentclass[11pt,leqno]{amsart}
\usepackage{enumerate}
\usepackage{amssymb,amsmath,amsthm,amsfonts}
\usepackage {latexsym}
\usepackage{bbm}
\usepackage{color}

\allowdisplaybreaks
\newcommand{\nc}{\newcommand}
\nc{\les}{\lesssim}
\nc{\nit}{\noindent}
\nc{\nn}{\nonumber}
\nc{\D}{\partial}
\nc{\diff}[2]{\frac{d #1}{d #2}}
\nc{\diffn}[3]{\frac{d^{#3} #1}{d {#2}^{#3}}}
\nc{\pdiff}[2]{\frac{\partial #1}{\partial #2}}
\nc{\pdiffn}[3]{\frac{\partial^{#3} #1}{\partial{#2}^{#3}}}
\nc{\abs}[1] {\lvert #1 \rvert}
\nc{\cAc}{{\cal A}_c}
\nc{\cE}{{\cal E}}
\nc{\cF}{{\cal F}}
\nc{\cP}{{\cal P}}
\nc{\cV}{{\cal V}}
\nc{\cQ}{{\cal Q}}
\nc{\cGin}{{\cal G}_{\rm in}}
\nc{\cGout}{{\cal G}_{\rm out}}
\nc{\cO}{{\cal O}}
\nc{\Lav}{{\cal L}_{\rm av}}
\nc{\cL}{{\cal L}}
\nc{\cB}{{\cal B}}
\nc{\cZ}{{\cal Z}}
\nc{\cR}{{\cal R}}
\nc{\cT}{{\cal T}}
\nc{\cY}{{\cal Y}}
\nc{\cX}{{\cal X}}
\nc{\cXT}{{{\cal X}(T)}}
\nc{\cBT}{{{\cal B}(T)}}
\nc{\vD}{{\vec \mathcal{D}}}
\nc{\efield}{\mathcal{E}}
\nc{\vE}{{\vec \efield}}
\nc{\vB}{{\vec \mathcal{B}}}
\nc{\vH}{{\vec \mathcal{H}}}
\nc{\ty}{{\tilde y}}
\nc{\tu}{{\tilde u}}
\nc{\tV}{{\tilde V}}
\nc{\Pc}{{\bf P_c}}
\nc{\bx}{{\bf x}}
\nc{\bX}{{\bf X}}
\nc{\bXYZ}{{\bf XYZ}}
\nc{\bY}{{\bf Y}}
\nc{\bF}{{\bf F}}
\nc{\bS}{{\bf S}}
\nc{\dV}{{\delta V}}
\nc{\dE}{{\delta E}}
\nc{\TT}{{\Theta}}
\nc{\dPsi}{{\delta\Psi}}
\nc{\order}{{\cal O}}
\nc{\Rout}{R_{\rm out}}
\nc{\eplus}{e_+}
\nc{\eminus}{e_-}
\nc{\epm}{e_\pm}
\nc{\eps}{\varepsilon}
\nc{\vnabla}{{\vec\nabla}}
\nc{\G}{\Gamma}
\nc{\w}{\omega}
\nc{\mh}{h}
\nc{\mg}{g}
\nc{\vphi}{\varphi}
\nc{\tlambda}{\tilde\lambda}
\nc{\be}{\begin{equation}}
\nc{\ee}{\end{equation}}
\nc{\ba}{\begin{eqnarray}}
\nc{\ea}{\end{eqnarray}}

\nc{\g}{\gamma}
\nc{\ol}{\overline}

\newtheorem{theorem}{Theorem}[section]
\newtheorem{lemma}[theorem]{Lemma}
\newtheorem{prop}[theorem]{Proposition}
\newtheorem{corollary}[theorem]{Corollary}
\newtheorem{defin}[theorem]{Definition}
\newtheorem{rmk}[theorem]{Remark}

\nc{\pT}{\partial_T}
\nc{\pz}{\partial_z}
\nc{\pt}{\partial_t}
\nc{\la}{\langle}
\nc{\ra}{\rangle}
\nc{\infint}{\int_{-\infty}^{\infty}}
\nc{\halfwidth}{6.5cm}
\nc{\figwidth}{10cm}
\newcommand{\f}{\frac}

\nc{\nlayers}{L} \nc{\nsectors}{M}
\nc{\indicator}{\mathbf{1}}
\nc{\Rhole}{R_{\rm hole}}
\nc{\Rring}{R_{\rm ring}}
\nc{\neff}{n_{\rm eff}}
\nc{\Frem}{F_{\rm rem}}
\nc{\R}{\mathbb R}
\nc{\Z}{\mathbb Z}
\nc{\DD}{\Delta}
\nc{\cD}{\mathcal D}
\nc{\lnorm}{\left\|}
\nc{\rnorm}{\right\|}
\nc{\rnormp}{\right\|_{\ell^{p,\eps}}}
\nc{\rar}{\rightarrow}
\date{\today}
\begin{document}

\begin{abstract}

Let $H=-\Delta+V$ be a Schr\"odinger operator on $L^2(\R^2)$ with real-valued potential $V$, and let $H_0=-\Delta$.  If $V$ has sufficient pointwise decay, the wave operators
$W_{\pm}=s-\lim_{t\to \pm\infty} e^{itH}e^{-itH_0}$ are
known to be bounded on $L^p(\R^2)$ for all $1< p< \infty$ if zero is not an eigenvalue or resonance.
We show that if there is an s-wave resonance or an eigenvalue only at zero, then the wave operators
are bounded on $L^p(\R^2)$ for $1 < p<\infty$.  This result stands in contrast to results in higher dimensions, where the presence of zero energy obstructions is known to shrink the range of valid exponents $p$.

\end{abstract}

\title[$L^p$ boundedness of wave operators]{\textit{On the $L^p$ boundedness of wave operators for two-dimensional Schr\"odinger operators with  threshold obstructions}}

\author[M.B.~Erdo\smash{\u{g}}an, M.~J. Goldberg, W.~R. Green]{M. Burak Erdo\smash{\u{g}}an, Michael Goldberg and William~R. Green}

\address{Department of Mathematics \\
University of Illinois \\
Urbana, IL 61801, U.S.A.}
\email{berdogan@math.uiuc.edu}
\address{Department of Mathematics \\
University of Cincinnati\\
Cincinnati, OH 45221-0025}
\email{Michael.Goldberg@uc.edu}
\address{Department of Mathematics\\
Rose-Hulman Institute of Technology \\
Terre Haute, IN 47803 U.S.A.}
\email{green@rose-hulman.edu}
%\subjclass{35Q41, 42B20}

\thanks{The first author was partially supported by the NSF grant  DMS-1501041. The second author
was  partially  supported  by  a  grant  from  the  Simons  Foundation (Grant  Number 281057.)  The third is supported by  Simons  Foundation  Grant  511825, and also acknowledges the support of a Rose-Hulman summer professional development grant.}

\maketitle

\section{Introduction}

Let $H=-\Delta+V$ be a Schr\"odinger operator
with a real-valued potential $V$ and
$H_0=-\Delta$.  If 
$|V(x)|\les \la x\ra^{-\beta}$ for some $\beta>2$, then
the spectrum of $H$ is composed of  a finite collection of
non-positive eigenvalues along with the absolutely continuous
spectrum on $[0,\infty)$, \cite{RS1}.
The wave operators are defined by the strong limits on $L^2(\R^n)$
\begin{align}
W_{\pm}f =\lim_{t\to\pm \infty} e^{itH}e^{-itH_0}f.
\end{align}
Such limits are known to exist and be asymptotically
complete  for a wide class of
potentials $V$.  
%That is, the image of $W_{\pm}$ is equal to the absolutely continuous subspace of $L^2(\R^n)$ associated to the Schr\"odinger operator $H$.  
Furthermore, one has
the identities 
\begin{align}\label{eq:intertwine}
W_\pm^* W_\pm=I, \qquad W_\pm W_\pm^*=P_{ac}(H),
\end{align}
with $P_{ac}(H)$ the projection onto the absolutely continuous
spectral subspace associated with the Schr\"odinger operator $H$.

This work continues a line of inquiry on the $L^p(\R^n)$ and $W^{k,p}(\R^n)$ boundedness of the wave operators.  It is known, see \cite{YajWkp1,YajWkp3,JY2,Yaj,FY,Bec,BS,BS2,DY} that the wave operators are bounded on $L^p(\R^n)$ for $1\leq p\leq \infty$ in dimensions $n\geq 3$ provided that zero energy is regular, with varying assumptions on the decay and smoothness of the potential $V$.

We say that zero energy is regular if there are no zero energy eigenvalues or resonances.
There is a zero energy eigenvalue if there is a
solution to $H\psi =0$ with $\psi\in L^2(\R^n)$, and a resonance if $\psi$ belongs to a function space
(depending on $n \leq 4$) strictly larger than $L^2(\R^n)$. 
%$\psi\notin L^2(\R^n)$, whose precise definition depends on the spatial dimension  $n\leq 4$. 
In dimension $n=2$, there is a rich structure of threshold obstructions.  If $H\psi=0$ with $\psi \in L^\infty (\R^2)\setminus L^2(\R^2)$ we say there is a zero energy resonance.  If $\psi \in L^\infty(\R^2)$ but $\psi \notin L^p(\R^2)$ for any $p<\infty$, we say that $\psi $ is an {\em s-wave resonance}.  If $\psi \in L^p(\R^2)$ for all $p>2$, we say that $\psi$ is a {\em p-wave resonance}.  We say there is a resonance of the first kind at zero if there is an s-wave resonance, but no p-wave resonance or eigenfunction at zero energy.

In dimensions $n\geq 3$, recent work of Yajima \cite{YajNew,YajNew3} and the second and third authors \cite{GGwaveop,GGwaveop4} show that zero-energy eigenvalues generically shrink the range of $L^p(\R^n)$ boundedness to $1< p<n$ if $n=3,4$ and $1< p<\frac{n}{2}$ if $n\geq 5$, with conditions on vanishing moments of the product of the potential and the zero-energy eigenfunctions allowing one to push the upper range to $1<p<\infty$.  In dimensions $n>4$ or $n=3$, one can obtain the $p=1$ endpoint, \cite{GGwaveop,YajNew3}.

The endpoints of $p=1$ and $p=\infty$ are elusive in lower dimensions.  Weder showed that the wave operators are bounded on $L^p(\R)$ for $1<p<\infty$, and that the endpoint $p=1$ is possible under certain conditions on the Jost solutions, but is weak-type in general \cite{Wed}.  See also \cite{DF}.   In two dimensions, Yajima showed that the  wave operators are bounded for $1<p<\infty$, when zero is regular and $\int_{\R^2} V(x)\, dx\neq 0$, \cite{Yaj2}.  The last hypothesis on $V$ was shown to be unneccessary in \cite{JY2}.
To the best of the
authors' knowledge, there are no results in the literature
when zero is not regular and $n=2$.

The intertwining identity
\begin{align}\label{eqn:intertwine}
	f(H)P_{ac}(H)=W_{\pm}f(-\Delta)W_{\pm}^*,
\end{align}
which is valid for Borel functions $f$, allows one to deduce properties of the perturbed operator $f(H)$ from the
simpler operator $f(-\Delta)$, provided one has
control on mapping properties of the wave operators
$W_\pm$ and $W_\pm^*$.
In two dimensions, boundedness of the 
wave operators on $L^p(\R^2)$ for a given $p\geq 2$ imply the dispersive estimates
\begin{equation} \label{eqn:dispersive}
	\|e^{itH}P_{ac}(H)\|_{L^{p'}\to L^p}\les
	|t|^{-1+\frac{2}{p}}.
\end{equation}
Here $p^\prime$ is the H\"older conjugate defined by
$\frac{1}{p}+\frac{1}{p^\prime}=1$.  
%In this way, one can use the $L^p$ boundedness of the wave operators to deduce dispersive estimates for the Schr\"odinger evolution.  

There has been much work on dispersive estimates
for the Schr\"odinger evolution with zero energy obstructions in recent years by  Schlag, Toprak and the authors in various combinations, see \cite{ES2,goldE,EG,EGG,GGodd,GGeven,Top} in which $L^1(\R^n)\to L^\infty(\R^n)$ were studied for all $n>1$.  Estimates in $L^p(\R^n)$ are obtained by interpolating these results with the  $L^2$ conservation law.  These works
have roots in previous work of \cite{JSS}, and also in
\cite{Jen2,Mur} where the dispersive estimates were studied as operators on weighted
$L^2(\R^n)$ spaces.

Our main results are inspired by the dispersive estimates proven in \cite{EG}.
In particular, it was shown that the existence of an s-wave resonance at zero does not destroy the natural $t^{-1}$ dispersive estimate.  It was further shown that in the case of an eigenvalue only, one can attain a $t^{-1}$ dispersive decay at the cost of polynomial spatial weights.  These estimates, along with the intertwining identity \eqref{eqn:intertwine} suggest that the wave operators should be $L^p(\R^2)$ bounded for a non-trivial range of $p$.  Our main result affirms this.

\begin{theorem}\label{thm:main}

	Assume that
	$|V(x)|\les \la x\ra^{-\beta}$.
	\begin{enumerate}[i)]
		\item \label{swave result}
		If there is an s-wave resonance, but no p-wave resonance or eigenvalue at zero, then
		the wave operators extend to bounded operators on $L^p(\R^2)$
		for all $1 < p < \infty$ provided $\beta>6$.

		\item \label{eval result} If  
		there is an eigenvalue at zero but no resonances,
		then the wave operators
		extend to bounded operators on $L^p(\R^2)$ for all
		$1 < p<\infty$ provided $\beta>12$.		
		
	\end{enumerate}
	
\end{theorem}

The near full range of $L^p(\R^2)$ boundedness is somewhat surprising due to the known results in higher dimensions in which further assumptions on the eigenspace are needed, \cite{GGwaveop,YajNew2,GGwaveop4}.  See also Remark~1.3 in \cite{Yaj2}.  The range for the eigenvalue only case utilizes orthogonality properties between the zero energy eigenfunctions and the potential.  In particular, we have for any zero energy eigenfunction $\psi$ that
$$
	\int_{\R^2} V\psi(x)\, dx=\int_{\R^2} x_j V\psi(x)\, dx=0, \qquad j=1,2.
$$
In higher dimensions, \cite{GGwaveop,YajNew2,GGwaveop4}, the addition of these vanishing moment assumptions are crucial to extending the range of $L^p(\R^n)$ boundedness to $1\leq p<\infty$.

We prove Theorem~\ref{thm:main} for $W=W_-$.
The proof for $W_+$ is identical up to complex conjugation.
The limiting resolvent operators are defined by $R_0^\pm(\lambda^2) := \lim\limits_{\eps \to 0^+} (H_0 - (\lambda \pm i\eps)^2)^{-1}$
and $R_V^\pm(\lambda^2) := \lim\limits_{\eps \to 0^+} (H - (\lambda \pm i\eps)^2)^{-1}$.  We refer to these operators
as the free and perturbed resolvents, respectively.  These operators are well-defined on polynomially
weighted $L^2(\R^2)$ spaces due to the limiting absorption
principle of Agmon, \cite{agmon}.

It is well-known that the free resolvent operator may be expressed in terms of special functions.  In particular, 
\be \label{eqn:R0bessel}
	R_0^\pm(\lambda^2)(x,y)=\pm \frac{i}{4}H_0^\pm (\lambda|x-y|),
\ee
where $H_0^\pm(z)=J_0(z)\pm i Y_0(z)$ are the Hankel functions of order zero, which are composed of the Bessel functions $J_0$ and $Y_0$.

The starting point for our analysis is the so-called ``stationary representation" of the wave operator,
\be \label{eq:statrep}
W^-u=u-\frac{1}{\pi i} \int_0^\infty \lambda R_V^-(\lambda^2) V[R_0^+-R_0^-](\lambda^2) u \, d\lambda
\ee
Due to the results of Yajima, \cite{Yaj2}, the high energy portion of the wave operator, when $\lambda>\lambda_0>0$ for any $\lambda_0\ll1$, is bounded on $L^p(\R^2)$ for $1<p<\infty$.  Accordingly, 
we are interested in the low energy contribution, so we insert the cut-off $\chi(\lambda)$ which is equal to one if $0<\lambda<\lambda_1$ and is zero if $\lambda>2\lambda_1$ for a small fixed constant $\lambda_1\ll 1$.

The paper is organized as follows.  We begin in Section~\ref{sec:res} developing the necessary low energy expansions of the perturbed resolvent operators $R_V^\pm(\lambda^2)$ in the presence of zero-energy obstructions.  In Section~\ref{sec:swave}, we prove that in the case of an s-wave resonance at zero energy, the most singular term may be bounded pointwise by an operator that is bounded on the full range of $1\leq p\leq \infty$.  In Section~\ref{sec:eval}, we show that in the case of an eigenvalue only at zero energy, the leading term is bounded on $1\leq p<\infty$.  In Section~\ref{sec:thmpf} we provide the proof of Theorem~\ref{thm:main}.  Finally, in Section~\ref{sec:ests}, we prove necessary technical integral estimates.

\section{Resolvent Expansions}\label{sec:res}

In this section we recall, and modify as needed, expansions for the low energy resolvent $R_V^\pm(\lambda^2)$ derived in \cite{EG}, see also \cite{JN,Sc2}.  These expansions were developed to study the $L^1(\R^2)\to L^\infty(\R^2)$ estimates, and require some modifications to suit the goal of established boundedness of the wave operators.

To analyze the low energy contribution of the wave operator we employ the symmetric resolvent identity for the perturbed resolvent operator.  We define $U(x)=1$ when $V(x)\geq 0$ and $U(x)=-1$ if $V(x)<0$, $v=|V|^{\f12}$.  Then $V=Uv^2$ and for $\Im(\lambda)\not=0$,
we have
$$
	R_V(\lambda^2)V=R_0(\lambda^2)vM(\lambda)^{-1}v,
	%\qquad 0<\Re (\lambda)<2\lambda_1 \ll 1,
$$
where  
$$M(\lambda)=U+vR_0(\lambda^2)v.$$  For $\lambda \in \R \setminus\{0\}$ define
$M^\pm(\lambda) = U + vR_0^\pm(\lambda^2)v$.   The invertibility of $M^\pm(\lambda)$ around $\lambda=0$ is intimately tied to the existence of obstructions (eigenvalues and/or resonances) at the threshold $\lambda=0$, see \cite{JN,EG}.  
So, inserting the resolvent identity into \eqref{eq:statrep}, our goal is to bound the integral kernel of
\begin{align}\label{eq:Minvlow}
	\frac{1}{\pi i} \int_0^\infty R_0^- v M^-(\lambda)^{-1}v[R_0^+-R_0^-](\lambda^2) \lambda \chi(\lambda)\, d\lambda.
\end{align}
%In particular, this means there's no Born series term to consider at low energy.
%To control the integral kernel of the wave operator in \eqref{eq:Minvlow}, we note the following bounds on $M^\pm(\lambda)^{-1}$ and related operators.

We note that expansions for $M^{\pm}(\lambda)^{-1}$ and hence of the resolvent $R_V^\pm(\lambda^2)$ in the presence of zero energy obstructions were a significant achievement in the work of Jensen and Nenciu \cite{JN}.  Further study was performed by the first and third authors, \cite{EG}, in service of establishing dispersive estimates.  From these works it is known that the singular behavior of $M^{\pm}(\lambda)^{-1}$ as $\lambda \to 0$ is highly dependent on the type of obstruction at zero energy.  We recall the following definitions from \cite{EG}, see also \cite{JN,Sc2}.

\begin{defin}
	We say an operator $T:L^2(\R^2)\to L^2(\R^2)$ with kernel
	$T(\cdot,\cdot)$ is absolutely bounded if the operator with kernel
	$|T(\cdot,\cdot)|$ is bounded from $L^2(\R^2)$ to $L^2(\R^2)$.
\end{defin}

It is worth noting that Hilbert-Schmidt and finite-rank operators are absolutely bounded.

We say that an absolutely bounded operator $T(\lambda)(\cdot,\cdot)$ is $\widetilde O_k(\lambda^s)$ if the integral
kernel satisfies 
\begin{align}\label{O def 1}
\big\|\sup_{0<\lambda<2\lambda_1} \lambda^{j-s}|\partial_\lambda^j T(\lambda)(\cdot,\cdot)|\big\|_{L^2\to L^2}\lesssim 1, \qquad 0\leq j \leq k.
\end{align}

Define $P$ to be the projection onto the span of $v$, and
$Q:=\mathbbm{1}-P$.

For small energies near the threshold, the expansion for the perturbed resolvent $R_V^\pm$ is found through expansions for the free resolvents $R_0^\pm$ both directly and through the operators $M^\pm(\lambda)^{-1}$.  In particular, we recall (see Section~3 of \cite{Sc2} for example), that the integral kernel of the free resolvents satisfy
$$
	R_0^\pm(\lambda^2)(x,y)=g^{\pm}(\lambda)+G_0(x,y)+E_0^\pm(\lambda)(x,y)
$$
Here, $E_0$ is an error term, and
\begin{equation} \label{eq:G0}
G_0f(x) = -\frac{1}{2\pi} \int_{\R^2} \log|x-y|f(y)\,dy,
\end{equation}
We note that $G_0$ is the fundamental solution to the Laplace equation on $\R^2$.  This naturally yields an expansion for the operators $M^{\pm}(\lambda)$.  With $T = U + vG_0v$, we have
$$
	M^{\pm}(\lambda)=g^{\pm}(\lambda)P+T+vE_0^\pm(\lambda)v.
$$
The nature of the threshold obstruction dictates the properties one needs from $E_0$.  We develop appropriate expansions in Lemma~\ref{lem:Minv swave} and \ref{M evalonlycor}.

\begin{defin}\label{resondef}\begin{enumerate}
		\item We say zero is a regular point of the spectrum
		of $H = -\Delta+ V$ provided $ QTQ=Q(U + vG_0v)Q$ is invertible on $QL^2(\mathbb R^2)$.
		
		\item Assume that zero is not a regular point of the spectrum. Let $S_1$ be the Riesz projection
		onto the kernel of $QTQ$ as an operator on $QL^2(\mathbb R^2)$.
		Then $QTQ+S_1$ is invertible on $QL^2(\mathbb R^2)$.  Accordingly, we define $D_0=(QTQ+S_1)^{-1}$ as an operator
		on $QL^2(\R^2)$.
		We say there is a resonance of the first kind at zero if the operator $T_1:= S_1TPTS_1$ is invertible on
		$S_1L^2(\mathbb R^2)$.  In this case, we define $D_1:=T_1^{-1}$ as an operator on $S_1L^2(\R^2)$.

	\end{enumerate}
	
\end{defin}

The projection $S_1$ is finite rank.  Further details on this inversion process and the related spectral analysis may be found in \cite{JN} and \cite[Section~5]{EG}.  We note further that the operator $QD_0Q$ is absolutely bounded.

If there is a resonance of the first kind at zero, the fact that the rank-one operator $T_1$ is invertible on
$S_1L^2(\R^2)$ immediately implies that $S_1$ is a rank-one projection and $D_1$ acts on a one-dimensional subspace of
$L^2(\R^2)$.  It is shown in~\cite{JN} that a resonance of the first kind occurs precisely when there is an s-wave
resonance at zero but no eigenfunctions or p-wave resonances at zero.  The bounded solution of $H\psi = 0$ can
be recovered directly from the image of $S_1$.

The following functions arise naturally in the expansion of the resolvent operators:
$$g^{\pm}(\lambda)=a\ln \lambda+z  \qquad a\in\R\backslash\{0\}, \quad  z\in\mathbb C\backslash\R.$$
These arise from the small argument expansion of the Bessel functions,  \cite{AS,Sc2,EG}.  We employ the notation $a-:=a-\epsilon$ for an arbitrarily small, but fixed $\epsilon>0$.  Similarly, $a+:=a+\epsilon$.

\begin{lemma}\label{lem:Minv swave}
	
	Assuming $|V(x)|\les \la x\ra^{-5-}$, and there is an s-wave resonance only at zero, then for a sufficiently small $\lambda_1>0$ and $0<\lambda<\lambda_1$, we have
	\begin{align}
		M^\pm(\lambda)&=g^\pm(\lambda)P+T+\widetilde O_2 (\lambda^{\f32+}),\\
		(M^\pm (\lambda)+S_1)^{-1}&=h_{\pm}(\lambda)^{-1}S+QD_0Q+\widetilde O_2 (\lambda^{\f32+}),\\
		M^\pm (\lambda)^{-1}&=-h_\pm (\lambda) S_1D_1S_1-SS_1D_1S_1-S_1D_1S_1S\\
		&\quad-h_\pm(\lambda)^{-1}SS_1D_1S_1S+h_{\pm}(\lambda)^{-1}S+QD_0Q+\widetilde O_2 (\lambda^{\f32+}).\nn
	\end{align}
	Here  $h_\pm (\lambda)=g^\pm(\lambda)+b$ for some $b\in \R\backslash \{0\}$.  Both $D_1$ and $S$ are finite-rank absolutely bounded operators on $L^2 (\R^2)$, with ${\rm Rank} (D_1) = 1$ and ${\rm Rank} (S) \leq 2$.
 
\end{lemma}

\begin{proof}
	
	The first bound is a simple adjustment of Lemma~2.2 in \cite{EG} taking $k=\f32+$ for the first derivative terms.  Since we need two derivatives of the error term, this requires the additional decay assumed on the potential.  The proof follows exactly as in \cite{EG} until one gets to the second derivative of the error term $M_0(\lambda)$.  In this case, recalling \eqref{eqn:R0bessel} and the expansion of the Bessel functions about infinity, we have
	$$
		|\partial_\lambda^2 R_0^\pm(\lambda^2)(x,y)\widetilde \chi(\lambda |x-y|)|\les \frac{|x-y|^{\f32}}{\lambda^{\f12}}.
	$$	
	The expansions of the Bessel functions used in \cite{EG} near zero require no further modifications.
	Once this bound is achieved the remaining bounds follow from the proofs of Lemma~2.5, Proposition~2.6 and Corollary~2.7 in \cite{EG}.
\end{proof}

Thus, in the case of an s-wave only at zero energy, we have the expansion for \eqref{eq:Minvlow}, and may express the low-energy portion of the wave operator as
\begin{multline}\label{eq:Minvlow2}
\frac{1}{\pi i} \int_0^\infty R_0^- v \bigg[-h_\pm (\lambda) S_1D_1S_1-SS_1D_1S_1-S_1D_1S_1S+QD_0Q\\
-h_\pm(\lambda)^{-1}SS_1D_1S_1S+h_{\pm}(\lambda)^{-1}S+\widetilde O_2 (\lambda^{\f32+})
\bigg]v[R_0^+-R_0^-](\lambda^2) \lambda \chi(\lambda)\, d\lambda.
\end{multline}

On the other hand, if there is an eigenvalue only at zero,

\begin{lemma}\label{M evalonlycor}
	
	Assume that there is an eigenvalue but no resonances at zero, and that $|V(x)|\les \la x\ra^{-12-}$.  Then
	\begin{multline}
	M^{\pm}(\lambda)^{-1}=\frac{S_3D_3S_3}{\lambda^2}+ (a_1\log\lambda+b_{1,\pm}) \Xi_1+\Big(1+\frac{b_{3,\pm}}{a_2\log\lambda+b_{2,\pm}}\Big)\Xi_2 \\
	+\frac{1}{h^{\pm}(\lambda)}\Xi_3+(M^{\pm}(\lambda)+S_1)^{-1}+ \widetilde O_2(\lambda^{1+}).
	\label{evalyonly Minv}
	\end{multline}
	Here, $S_3D_3S_3$ is a finite-rank operator, $\Xi_i$ are real-valued absolutely bounded operators, $\Xi_2$ and $\Xi_3$ have a projection orthogonal to
	$P$ on at least one side, and $\Xi_1$ have orthogonal projections on both sides. Further $a_i\in\R\setminus\{0\}$ and
	$b_{i,+}=\overline{b_{i,-}}$.
	
\end{lemma}

This expansion, with a slightly different error term, is found in \cite[Corollary~6.2]{EG} to ensure that the error terms is amenable to the full range of $p$.  This requires only slightly more decay on the potential.  Define $g_2^\pm(\lambda)=\lambda^4(a_2\log\lambda+b_{2,\pm})$ and $g_3(\lambda)=a_3\lambda^4$
with $a_2, a_3\in \R\setminus\{0\}$ and $b_{2,-}=\overline{b_{2,+}}$. Also let
$G_j$ be integral operators with the kernel $|x-y|^{j+1}$ for $j=1,3$, and if $j=2,4$, $G_j$ has kernel is
$|x-y|^{j}\log |x-y|$.
\be\label{M0further}
M_0^{\pm}(\lambda)=g_1^{\pm}(\lambda)vG_1v+  \lambda^2 vG_2v
+g_2^{\pm}(\lambda)vG_3v+g_3(\lambda) vG_4v+\widetilde O_2(\lambda^{5+}),
\ee
by expanding the Bessel functions to order $z^{6}\log z$ and estimating the error term as
in Lemma~2.2 in \cite{EG}. This extra $\lambda$ smallness in the error term is paired with a spatial growth of size $v(x)|x-y|^{5+}v(y)$, which necessitates that $|V(x)|\les \la x\ra^{-12-}$ to be Hilbert-Schmidt.

We note that the only the most singular $\lambda^{-2}$ term is new in this case.  The $\Xi_1$ term is entirely analogous to the `s-wave' term $S_1D_1S_1$ in the expansion in Lemma~\ref{lem:Minv swave} with respect to the spectral parameter $\lambda$ and the orthogonality properties that we use.  The operator with the $\lambda^{-2}$ term is similar to the eigenvalue term encountered in higher dimensions \cite{GGwaveop,GGwaveop4} and has orthogonality properties that we exploit to prove an expanded range of $L^p(\R^2)$ boundedness.

We show that many of the terms that arise in this expansion have an integral kernel that is admissible.  We say that an operator $K$ with integral kernel $K(x,y)$ is admissible if
$$
\sup_{x\in \R^n} \int_{\R^n } |K(x,y)|\, dy+\sup_{y\in \R^n} \int_{\R^n } |K(x,y)|\, dx<\infty.
$$
It is well-known that an operator with an admissible kernel is bounded on $L^p(\R^n)$ for all $1\leq p\leq \infty$.  In addition, we use the following lemma whose proof is in Section~\ref{sec:ests}.

\begin{lemma}\label{lem:Lp kernel}
	
	If $K$ is an integral operator whose kernel satisfies the pointwise estimate
	$$
		|K(x,y)|\les \frac{1}{\la x\ra \la |x|-|y|\ra^2} \qquad \text{or} \qquad \frac{1}{\la x\ra^{1-\epsilon} \la |x|-|y|\ra \la |x|+|y|\ra}
	$$
	for any $0<\epsilon<1$, 
	then $K$ is bounded on $L^p(\R^2)$ for all $1\leq p<\infty$.
	
\end{lemma}

%It looks like Kenji has controlled the first (non-s-wave) terms all as operators on $1<p<\infty$ in his paper and the follow-up with Jensen.  It seems the $QD_0Q$ term  ought to reach endpoints, but this likely requires the use of the orthogonality trick.

\section{The s-wave only case}\label{sec:swave}

In this section our goal is to show the $L^p(\R^2)$ boundedness of the most singular term that arises in the expansion \eqref{eq:Minvlow2}.  In contrast to the treatment of the analogous terms in the dispersive estimate treatment, \cite{EG}, we employ an integral formulation of the Mean Value Theorem which utilizes the orthogonality porperties of the leading operator.  This allows us to gain faster pointwise decay of the integral kernel, which leads to the extended range of $L^p$-boundedness.   To that end, we prove

\begin{prop}\label{prop:swave prop}
	
	The operator defined by
	\begin{align}\label{eq:wave op}
	Au=
	\frac{1}{\pi i} \int_0^\infty h^-(\lambda) R_0^-(\lambda^2)v S_1D_1S_1v\big[R_0^+(\lambda^2)-R_0^-(\lambda^2) \big] \lambda \chi(\lambda) u \, d\lambda
	\end{align}
	may be extended to a bounded operator on $L^p(\R^2)$ for any $1\leq p\leq \infty$, provided that $v(x)\les \la x\ra^{-3-}$.
	
\end{prop}

Here we take advantage of the fact that the projection operator $  S_1\leq Q$ has the orthogonality property that $S_1v1=0$.  In addition, using \eqref{eqn:R0bessel}, we see that
\be\label{vital}
	R_0^+(\lambda^2)(x,y)-R_0^-(\lambda^2)(x,y)=\frac{i}{2}J_0(\lambda|x-y|).
\ee
This observation is vital as the small argument behavior of $J_0$ is  better than that of $Y_0$ in the spectral parameter $\lambda$.
We need to consider
\begin{equation} \label{eqn:sFirstTerm}
\int_0^\infty \iint_{\R^4} H_0^{-}(\lambda|x-z|)[v S_1D_1S_1v](z,w) J_0(\lambda |w-y|)   h^-(\lambda) \lambda\chi(\lambda)\,dwdz\,d\lambda.
\end{equation}
Using the orthogonality conditions, we can replace $H_0^-(\lambda|x-z|)$ in~\eqref{eqn:EvalueFirstTerm}
with 
\begin{equation} \label{eqn:SubtractingZeroH}
H_0^-(\lambda|x-z|) - H_0^-(\lambda \la x\ra) = \lambda \int_{\la x\ra}^{|x-z|} (H_0^-)'(\lambda r)\,dr,
\end{equation}
and replace $J_0(\lambda|y-w|)$ with
\begin{equation} \label{eqn:SubtractingZeroJs}
J_0(\lambda|y-w|) - J_0(\lambda |y|)    
=   \lambda \int^{|y-w|}_{|y|}  J_0'(\lambda s)\,ds.
\end{equation}
As a result, we  make use the following oscillatory integral estimate, whose proof  we postpone to Section~\ref{sec:ests}.
\begin{lemma} \label{lem:J's}
	For fixed constants $r,s>0$, we have the bound
\begin{equation*}
\Big|\int_0^\infty (H_0^-)'(\lambda r) J_0'(\lambda s) \lambda^{3 }h^-(\lambda) \chi(\lambda)\,d\lambda \Big| 
\les k(r,s),
\end{equation*}
where 
$$
k(r,s):= \frac{1 }{ \sqrt{rs}  \la r - s\ra^2}+\frac1{r\la r+ s\ra^{2+}}.
$$
\end{lemma}
With this estimate, we can prove the main proposition of this section.

\begin{proof}[Proof of Proposition~\ref{prop:swave prop}]
Substituting \eqref{eqn:SubtractingZeroH} and \eqref{eqn:SubtractingZeroJs}
into~\eqref{eqn:sFirstTerm}, the result is
\begin{equation}\label{eqn:HprimeJprimes}
\int_0^\infty \iint_{\R^4}\int_{\la x\ra}^{|x-z|} \int^{|y-w|}_{|y|}  (H_0^{-})'(\lambda r)[v S_1D_1S_1v](z,w) J_0' (\lambda s) 
\lambda^{3 }h^-(\lambda) \chi(\lambda)\,ds dr dwdz d\lambda.
\end{equation}
Note that we can change the order of integration provided that  $|v(x)|\les \la x\ra^{-\frac32-}$.
Evaluating the $\lambda$ integral first using Lemma~\ref{lem:J's}
the resulting expression is bounded by
\begin{equation}\label{tempor}
\iint_{\R^4}\int_{\la x\ra}^{|x-z|} \int^{|y-w|}_{|y|}k(r,s) [v S_1D_1S_1v](z,w)
\,ds dr dwdz.
\end{equation}
Let  $T(z,w)=\la z\ra^{N}\la w\ra^{N} [v S_1D_1S_1v](z,w)$, with $N=2+$. Note that $|T|$ is integrable in $z$ and $w$ by the absolute boundedness of $S_1D_1S_1$ provided that $|v(x)|\les \la x\ra^{-3- }$, which suffices to ensure that  $\la \cdot \ra^{N}v \in L^2$.  
Interchanging the order of integration of 
$s$ and $w$ yields
\begin{multline*}
 |\eqref{tempor}|\les \int_{\R^2} \int_{\la x\ra}^{|x-z|}\int_0^{|y|} \int_{|y-w|<s} k(r,s) \la z\ra^{-N} \la w\ra^{ -N} |T(z,w)|  
 \,dw ds dr dz \\
+ \int_{\R^2} \int_{\la x\ra}^{|x-z|}\int_{|y|}^\infty \int_{|y-w|>s}  k(r,s) \la z\ra^{-N} \la w\ra^{ -N} |T(z,w)| 
\,dw ds dr dz \\
\leq \int_{\R^2} \int_{\la x\ra}^{|x-z|}\int_0^{\infty} \int_{|w| \geq |s-|y||}  k(r,s) \la z\ra^{-N} \la w\ra^{ -N} |T(z,w)| 
 \,dw ds dr dz \\
\les \int_{\R^2} \int_{\la x\ra}^{|x-z|}\int_0^{\infty} \int_{\R^2}  \frac{ k(r,s)  \la z\ra^{-N}   |T(z,w)|}{ \la s-|y|\ra^N} 
 \,dw ds dr dz.
\end{multline*}
Repeating the same argument with $z$ and $r$ integrals, noting that $r>\min(1,|x-z|)$, and selecting $0<\epsilon<\f12$, we obtain the bound
\begin{multline*}
 \int_0^1 \int_0^{\infty}  \iint_{\R^4}  \frac{ k(r,s) r^{\epsilon}(1 +|x-z|^{ -\epsilon}) |T(z,w)|}{  \la r-|x|\ra^N\la s-|y|\ra^N} 
 \,dz dw ds dr\\
+ \int_1^\infty \int_0^{\infty}  \iint_{\R^4}  \frac{k(r,s)  |T(z,w)|}{  \la r-|x|\ra^N\la s-|y|\ra^N} 
 \,dz dw ds dr.
\end{multline*}
Evaluating   the $w$ and $z$ integrals (noting that $\int (1+|x-z|^{ -\epsilon}) |T(z,w)| dz dw \les 1$), we obtain the bound
\be\label{kbound}
	|\eqref{tempor}| \les \int_0^\infty \int_0^{\infty}    \frac{ r^{\epsilon} k(r,s)  }{\la r \ra^\epsilon   \la r-|x|\ra^N\la s-|y|\ra^N} 
 \,ds dr.
\ee
Note that this is an admissible kernel since $N=2+$. Noting that
$$
	\int_{\R^2} \frac{1}{\la r-|x|\ra^{N}}\, dx \les \la r\ra,
$$
with the growth in $r$ 
due to the contribution on the annulus $r-1\leq |x|\leq r+1$ when $r$ is large.  For \eqref{kbound}, the $L^1_x$ integral is bounded by 
\begin{multline*}
 \int_0^\infty \int_0^{\infty}    \frac{ \la r\ra^{1-\epsilon}  r^{\epsilon} k(r,s)  }{    \la s-|y|\ra^N} 
 \,dr ds = \int_0^\infty \int_0^{\infty}  \Big(\frac{\la r\ra   }{ \sqrt{ r s}  \la r - s\ra^2}+\frac{\la r\ra^{1-\epsilon}  r^{\epsilon-1}}{ \la r+ s\ra^{2+}}\Big)  \frac{1    }{   \la s-|y|\ra^N} 
 \,dr ds 
 \\ 
% \lesssim \int_0^\infty \int_0^{\infty}  \Big(\frac{1}{    \la r - s\ra^2}+\frac1{\sqrt{s} \la r- s\ra^{\frac32}}\Big)  \frac{1    }{   \la s-|y|\ra^N} \,dr ds\\
 \les  \int_0^{\infty}   \frac{1+s^{-\frac12}}{   \la s-|y|\ra^N} 
 \,  ds\lesssim 1.
\end{multline*}
This bound is uniform in $y\in \R^2$.  
Similarly, the $L^1_y$ integral is bounded by
\begin{multline*}
 \int_0^\infty \int_0^{\infty}    \frac{ \la  s \ra r^{\epsilon} k(r,s) }{ \la r \ra^\epsilon  \la r-|x|\ra^N} 
 \,dr ds\\ 
 = \int_0^\infty \int_0^{\infty}  \Big(\frac{\la  s \ra }{ \sqrt{ rs}  \la r - s\ra^2}+\frac{\la  s \ra  }{ \la r \ra^\epsilon r^{1-\epsilon}\la r+ s\ra^{2+}}\Big)  \frac{1    }{   \la r-|x|\ra^N} 
 \,ds dr  \\ 
% \lesssim \int_0^\infty \int_0^{\infty}  \Big(\frac{1}{    \la r - s\ra^2}+\frac1{\sqrt{r} \la r- s\ra^{\frac32}}+\frac1{r^{1-\epsilon}\la   s\ra^{1+}}\Big)  \frac{1    }{   \la r-|x \ra^N}  \,ds dr  \\
 \les  \int_0^{\infty}   \frac{1+r^{\epsilon-1}}{   \la r-|x|\ra^N} 
 \,  dr\lesssim 1.
\end{multline*}
This bound is uniform in $x\in \R^2$.  Hence, the kernel is admissible which finishes the proof.
\end{proof}

It is easy to see that argument above suffices to bound the $QD_0Q$ term  as well, which is slightly better behaved in the spectral variable $\lambda$.

\begin{corollary}\label{cor:QDQ}
	
	The operator defined by
	\begin{align}\label{eq:QDQ}
	Au=
	\frac{1}{\pi i} \int_0^\infty  R_0^-(\lambda^2)v QD_0Qv\big[R_0^+(\lambda^2)-R_0^-(\lambda^2) \big] \lambda \chi(\lambda) u \, d\lambda
	\end{align}
	may be extended to a bounded operator on $L^p(\R^2)$ for any $1\leq p\leq \infty$, provided that $v(x)\les \la x\ra^{-3-}$.
	
\end{corollary}

\section{Eigenvalue Only}\label{sec:eval}

In this section we show that the kernel of the leading singular term in the resolvent expansion when there is only an eigenvalue at zero is pointwise bounded by a kernel which is bounded on $L^p(\R^2)$ for $1\leq p<\infty$.  As in the case of s-wave resonance, we choose to utilize the   orthogonality relationship $S_1v=0$ in an integral formulation of the Mean Value Theorem.  Noting the expansion for $M^\pm(\lambda)^{-1}$ in Lemma~\ref{M evalonlycor}, the most singular term involves a singularity of size $\lambda^{-2}$ as $\lambda \to 0$, while the remaining terms have analogous counterparts in the case of an s-wave resonance.  In this section we prove

\begin{prop}\label{prop:D3}
	
	The operator defined by
	\begin{align}\label{eq:D3}
	Au=
	\frac{1}{\pi i} \int_0^\infty  R_0^-(\lambda^2)v S_3D_3S_3v\big[R_0^+(\lambda^2)-R_0^-(\lambda^2) \big] \lambda^{-1} \chi(\lambda) u \, d\lambda
	\end{align}
	may be extended to a bounded operator on $L^p(\R^2)$ for any $1\leq p< \infty$, provided that $v(x)\les \la x\ra^{-4-}$.
	
\end{prop}

Once again we take advantage of the fact that the projection operator $S_3\leq S_1\leq Q$ has the orthogonality property that $S_3v1=0$.   
In addition, since $D_3$ acts on the finite dimensional space $S_3L^2$, and using 
\eqref{eqn:R0bessel} and \eqref{vital}, there will be a finite number of terms to consider of the form
\begin{equation} \label{eqn:EvalueFirstTerm}
\int_0^\infty \iint_{R^4} H_0^{-}(\lambda|x-z|) v\phi(z) v\psi(w) J_0(\lambda |w-y|) \lambda^{-1}\chi(\lambda)\,dwdz\,d\lambda,
\end{equation}
where $\phi, \psi \in S_3L^2(\R^2)$.  See Section~5 of \cite{EG} for further details on the spectral subspaces of $L^2(\R^2)$ associated to the zero-energy obstructions.
The singularity $\lambda^{-1} H_0^{-}(\lambda|x-z|)$ is not integrable at the $\lambda = 0$ endpoint. Therefore, as in the proof of Proposition~\ref{prop:swave prop} we must  
take advantage of the cancellation condition $\int_{\R^2} v\psi(w)\,dw = 0$ just to show that the resulting operator kernel
is finite anywhere.  The additional cancellation condition $\int_{\R^2} w_j v\psi(w)\,dw = 0$, $j = 1,2$, will be needed
to show that it is bounded on $L^p(\R^2)$ for $1 \leq p < \infty$.
With these orthogonality conditions, we can replace $H_0^-(\lambda|x-z|)$ in~\eqref{eqn:EvalueFirstTerm}
with  \eqref{eqn:SubtractingZeroH} 
and replace $J_0(\lambda|y-w|)$ with
\begin{multline} \label{eqn:SubtractingZeroJ}
J_0(\lambda|y-w|) - J_0(\lambda|y|) +\lambda\frac{w\cdot y}{|y|}J_0'(\lambda|y|) \\
=  \lambda(|y-w| - |y| + {\scriptstyle \frac{w\cdot y}{|y|}})J_0'(\lambda|y|) +
\lambda^2 \int_{|y-w|}^{|y|} (s-|y-w|)J_0''(\lambda s)\,ds.
\end{multline}

We will make use of oscillatory integral estimates similar to Lemma~\ref{lem:J's}, whose proofs we also postpone to Section~\ref{sec:ests}.  
\begin{lemma} \label{lem:J''}
	For fixed constants $r,s>0$, we have the bound
\begin{equation*}
\Big|\int_0^\infty (H_0^-)'(\lambda r) J_0''(\lambda s) \lambda^2 \chi(\lambda)\,d\lambda \Big|
\les k_2(r,s),
\end{equation*}
where 
$$
k_2(r,s):=\frac{1}{\sqrt{rs}\la r-s\ra^{2-}}+ \frac{1  }{r \la r+s\ra^2}.
$$
\end{lemma}

We also need another bound for the first derivative with one less power of $\lambda$.
\begin{lemma} \label{lem:J'}
	For fixed constants $r,s>0$, we have the bound
\begin{equation*}
\Big|\int_0^\infty (H_0^-)'(\lambda r) J_0'(\lambda s) \lambda \chi(\lambda)\,d\lambda\Big| 
\les \frac{s\la \log \la r\ra  \ra }{r \la r+s\ra \la r - s\ra}.
\end{equation*}
\end{lemma}

With these estimates, we are now ready to prove the main proposition of this section.

\begin{proof}[Proof of Proposition~\ref{prop:D3}]

When we substitute~\eqref{eqn:SubtractingZeroH} and the second term on the right side of~\eqref{eqn:SubtractingZeroJ}
into~\eqref{eqn:EvalueFirstTerm}, the result is
\begin{equation}\label{eqn:HprimeJdprime}
\int_0^\infty \iint_{\R^4}\int_{\la x\ra}^{|x-z|} \int_{|y-w|}^{|y|}(s-|y-w|) (H_0^{-})'(\lambda r) v\phi(z) v\psi(w) J_0''(\lambda s) 
\lambda^2 \chi(\lambda)\,ds dr dwdz d\lambda.
\end{equation}
This is now integrable at $\lambda = 0$, and so long as $|v\psi(w) |\les \la w\ra^{-4-}$ we may
change the order of integration to evaluate the $\lambda$ integral first.  It is also clear from the domain that $\big|s - |y-w|\big| \leq \big||y| - |y-w|\big| \leq |w|$. For similar reasons, $\big|s- |y|\big| \leq |w|$ as well. According to Lemma~\ref{lem:J''}
the resulting expression is bounded by
\begin{equation*}
\iint_{\R^4}\int_{\la x\ra}^{|x-z|} \int_{|y-w|}^{|y|} k_2(s,r)  |v\phi(z) v\psi(w)| |w|
\,ds dr dwdz.
\end{equation*}
Proceeding as in the proof of Proposition~\ref{prop:swave prop} (see the estimate \eqref{kbound} of \eqref{tempor}),  we bound the integral above by
\begin{multline}\label{tempor2}
\int_{0}^{\infty}\int_{0}^{\infty}  \frac{ r^{\epsilon} k_2(s,r)}{  \la r\ra^{\epsilon} \la s- |y|\ra^{2+ } \la r - |x|\ra^{3+ } } 
\,ds dr\\ \les \int_{0}^{\infty}\int_{0}^{\infty} \frac{1}{\sqrt{rs}\la r-s\ra^{2-}\la s- |y|\ra^{2+} \la r - |x|\ra^{3+} } \,ds dr\\ + \int_{0}^{\infty}\int_{0}^{\infty} \frac{1  }{r^{1-\epsilon}\la r\ra^{\epsilon} \la r+s\ra^2\la s- |y|\ra^{2+ } \la r - |x|\ra^{3+ } } \,ds dr,
\end{multline}
provided that $|v\psi(w) |\les \la w\ra^{-5-}$. We claim that the first integral in the right hand side of \eqref{tempor2} gives an admissible kernel. Indeed, as in the bound for \eqref{kbound}, the $L^1_y$ integral is bounded by
\begin{equation*}
 \int_{0}^{\infty}\int_{0}^{\infty}  \frac{\la s\ra }{    \sqrt{rs}\la r-s\ra^{2-}\la r - |x|\ra^{3+}  } 
\,ds dr 
\les \int_0^\infty \frac{1+r^{-\frac12}}{ \la r-|x|\ra^{3+}}\,dr
\les 1
\end{equation*}
uniformly in $x$.
Similarly, the $L^1_x$ integral is bounded uniformly in $y$.   

Using Lemma~\ref{lem:bracket decay} in the $s$ variable in the second integral on the right hand side of \eqref{tempor2} gives the bound
$$
\int_{0}^{\infty}  \frac{1  }{r^{1-\epsilon}\la r\ra^{\epsilon} \la r-|y|\ra^2  \la r - |x|\ra^{3+ } }  dr   
\les \int_{0}^{\infty}  \frac{1  }{r^{1-\epsilon}\la r\ra^{\epsilon} \la r-|x|\ra^{1+}\la |x|-|y|\ra^2 } dr.
$$
The last bound follows from $ \la r-|y|\ra  \la r - |x|\ra  \gtrsim\la|x|-|y|\ra$. Finally, using a simple variant of  Lemma~\ref{lem:bracket decay}, noting that the region where $r<1$ produces a better bound, we obtain the bound
$$
  \frac1{ \la  x \ra \la|x|-|y|\ra^2 }.
$$
This is a bounded operator in $L^p$ for $1\leq p<\infty$ by Lemma~\ref{lem:Lp kernel}. 

If we instead substitute the first term on the right side of~\eqref{eqn:SubtractingZeroJ} into the
same expression, the result is
\begin{equation*}
\int_0^\infty \iint_{\R^4}\int_{\la x\ra}^{|x-z|} (|y-w| - |y| + {\scriptstyle \frac{w\cdot y}{|y|}}) (H_0^{-})'(\lambda r) v\phi(z) v\psi(w) J_0'(\lambda |y|) 
\lambda \chi(\lambda)\, dr dwdz d\lambda.
\end{equation*}
Once again it is permissible to change the order of integration and evaluate the $d\lambda$ integral first.
Lemma~\ref{lem:J'} provides the bound
$$
\iint_{\R^4}\int_{\la x\ra}^{|x-z|} 
\frac{\la \log \la r\ra \ra (|y-w| - |y| + {\scriptstyle \frac{w\cdot y}{|y|}}) |y|}{r\la r+|y|\ra \la r-|y|\ra} v\phi(z) v\psi(w) \, dr dwdz.
$$
Note that $| |y-w| - |y| + {\scriptstyle \frac{w\cdot y}{|y|}}| |y| \les |w|^2$ by considering cases $|y|\les |w|$ and $|y|\gg |w|$ separately. Using this bound and proceeding as in the proof of Proposition~\ref{prop:swave prop} by changing the order of  $z$ and $r$ integrals, and then evaluating the integrals in $z$ and $w$, we obtain the bound
$$
\int_0^\infty \frac{\la \log \la r\ra \ra  }{\la r\ra^{\epsilon} r^{1-\epsilon}\la r+|y|\ra \la r-|y|\ra\la r-|x|\ra^{3+}}
\,  dr, 
$$
provided that $|v(x)|\les \la x\ra^{-5-}$. Using $\la r\pm|y|\ra \la r-|x|\ra\gtrsim \la |x|\pm |y|\ra $ and then using  a simple variant of Lemma~\ref{lem:bracket decay} yields the bound 
$$
\frac1{\la |x|-|y|\ra\la |x|+|y|\ra}\int_0^\infty \frac{\la \log \la r\ra \ra  }{\la r\ra^{\epsilon} r^{1-\epsilon} \la r-|x|\ra^{1+} }
\,  dr\les \frac1{\la x\ra^{1-} \la |x|-|y|\ra\la |x|+|y|\ra}.
$$
This makes it a bounded operator on
$L^p(\R^2)$, $1 \leq p < \infty$ by Lemma~\ref{lem:Lp kernel}
\end{proof}

\section{Proof of Theorem~\ref{thm:main}}\label{sec:thmpf}

In this section we combine the results of the previous sections to prove Theorem~\ref{thm:main}.  In addition to the results established in Sections~\ref{sec:res} and \ref{sec:eval}, we need to control the error term in our expansions.
The following lemma is a modification of Lemma~4.1 in \cite{Yaj2}.
\begin{lemma}\label{lem:errorbound}
	
	Assume that $\ell>1$ is fixed and $N(\lambda)$ is an operator that satisfies
	$$
	\left\| \frac{d^j}{d\lambda^j} N(\lambda)  \right\|_{L^2\to L^2} \les \lambda^{\ell-j}, \qquad j=0,1,2, \qquad 0<\lambda <2\lambda_1.
	$$
	Then the operator $A$ defined by
	\be\label{eqn:Aop defn}
	Au=-\frac{1}{\pi i} \int_0^\infty R_0^-(\lambda^2)v N(\lambda)v\big[R_0^+(\lambda^2)-R_0^-(\lambda^2) \big] \lambda \chi(\lambda) u \, d\lambda
	\ee
	may be extended to an operator bounded on $L^p(\R^2)$ for any $1\leq p\leq \infty$, provided that $v(x)\les \la x\ra^{-2-}$.
	
\end{lemma}

In other words, if $\chi(\lambda)N(\lambda)=\widetilde O_2(\lambda^\ell)$ as an operator on $L^2$, then \eqref{eqn:Aop defn} is bounded on $1\leq p\leq \infty$.

\begin{proof}
	
	Define the functions
	\begin{align*}
	G^\pm_y(y_1)=e^{\mp i \lambda |y|} R_0^\pm (\lambda^2)(y_1,y).
	\end{align*}
	Then, noting (3.4) in \cite{Yaj2}, we have the bounds
	\begin{align}\label{eq:Gbounds}
	\left|\frac{\partial^j}{\partial \lambda^j} G^\pm_y(y_1)  \right|
	\les \frac{\la y_1 \ra^j}{\sqrt{\lambda |y_1-y|}}.
	\end{align}
	Then, the $\lambda$ integral in \eqref{eqn:Aop defn} may be written as
	\be \label{eqn:A2}
	\int_0^\infty e^{-i\lambda (|x|\pm |y|)} a_{x,y}^\pm (\lambda) \chi(\lambda)\, d\lambda
	\ee
	where
	$$
	a_{x,y}(\lambda):= \lambda \la N(\lambda)v G_y^\pm,v G_x^- \ra.
	$$
	By assumption on $N(\lambda)$ and \eqref{eq:Gbounds}, we have 
	$$
	|\partial_\lambda^j a_{x,y}^\pm(\lambda)| \les \frac{\lambda^{\ell-j}}{(\la x\ra \la y \ra)^{\f12}}.
	$$
	Here we used that decay of $v$ along with the fact that
	$$
	\left\| \frac{\la x_1\ra^{-1-}}{|x-x_1|^{\f12}} \right\|_{L^2_{x_1}}\les \frac{1}{\la x\ra^{\f12}}.
	$$
	This  yields the bound
	$$
	|\eqref{eqn:A2}|\les \frac{1}{(\la x\ra \la y \ra)^{\f12}} \int_0^{2\lambda_1} \lambda^\ell \, d\lambda \les \frac{1}{(\la x\ra \la y \ra)^{\f12}}.
	$$	
	Now, using the smallness of $a_{x,y}^\pm(\lambda)$ and $\partial_\lambda a_{x,y}^\pm(\lambda)$ as $\lambda \to 0$ and the support of the cutoff function, we may integrate by parts twice without boundary terms to see
	$$
	|\eqref{eqn:A2}|\les \frac{1}{(|x|\mp |y|)^2(\la x\ra \la y \ra)^{\f12}} \int_0^{2\lambda_1} \lambda^{\ell-2} \, d\lambda \les \frac{1}{(|x|\mp |y|)^2(\la x\ra \la y \ra)^{\f12}}.
	$$
	Since $\ell>1$, $\lambda^{\ell-2}$ is integrable in a neighborhood of zero.
	These two bounds yield that
	$$
	|\eqref{eqn:A2}|\les \frac{1}{\la |x|\mp |y|\ra^2(\la x\ra \la y \ra)^{\f12}}
	$$	
	which is an admissible kernel. 
\end{proof}

\begin{proof}[Proof of Theorem~\ref{thm:main}]
	
Due to the high-energy wave operator bounds established in Sections~2 and 3 of \cite{Yaj2}, we need only prove the $L^p$ boundedness for small energies.  
We prove Theorem~\ref{thm:main} Part~\ref{swave result}) first.  Recall the expansion of the low-energy contribution to the wave operator in \eqref{eq:Minvlow2}.  The leading order term involves the operator $-h_{\pm}(\lambda)S_1D_1S_1$ is shown to be bounded on the full range $1\leq p\leq \infty$ in Proposition~\ref{prop:swave prop}.  Similarly the contribution of the operator $QD_0Q$ is bounded on the full range as well due to Corollary~\ref{cor:QDQ}.  The remaining terms involving the operator $S$ are bounded on the range $1<p<\infty$ in Section~2.2 in \cite{JY2}.  Finally, the error term is bounded on the full range $1\leq p\leq \infty$ by
Lemma~\ref{lem:errorbound}.

For the Part~\ref{eval result}, we employ the low-energy expansion of the wave operator obtained by inserting the expansion of Lemma~\ref{M evalonlycor} into \eqref{eq:Minvlow}.  The majority of the terms, all except the leading $\lambda^{-2}$ term, can be bounded similar to their counterparts already bounded in Part~\ref{swave result}.  The contribution of the remaining most singular term was shown to extend to a bounded operator if $1\leq p<\infty$ in Proposition~\ref{prop:D3}.  Again, the error term may be controlled by Lemma~\ref{lem:errorbound}, completing the proof. 
\end{proof}

\begin{rmk}
	
	We note that these techniques would allow for a slight improvement in Theorem~1.2 of \cite{EG}.  In particular, one can obtain the $t^{-1}$ time decay rate as an operator from $L^{1,0+}\to L^{\infty, 0-}$ in the case of an eigenvalue only at zero.  This removes one power of spatial weight from both spaces.
	
\end{rmk}

Our analysis does not seem to be immediately applicable to showing $L^p$ boundedness when there is a p-wave resonance at zero.   In Section~4 of \cite{GGwaveop4}, the difficulties inherent in the four dimensional resonance are discussed in detail.  Due to the similarities to the two dimensional p-wave resonance, many of the technical issues that provide a challenge to our pointwise bound approach remain.  If the wave operators were $L^p$ bounded for any $p>2$, it would imply a polynomial time decay of size $|t|^{(2/p)-1}$ as an operator on $L^p\to L^{p'}$ due to the intertwining identity \eqref{eqn:intertwine}.  However, the dispersive estimate in \cite{EG} or the weighted $L^2$ estimate in \cite{Mur}, along with a detailed analysis as in \cite{JY4,EGG}, shows that the leading term in the dispersive bound can decay no faster than $(\log t)^{-1}$ for large $t$.  The even-dimensional resonances are not well understood, while in three dimensions recent work of Yajima, \cite{YajNew3}, shows that the wave operators are bounded if and only if $1<p<3$ in the presence of a threshold resonance in three spatial dimensions.

The endpoints of $p=1$ and $p=\infty$ provide a serious technical challenge.  Even in the case when zero is regular, the low energy expansion is only known to be bounded on $1<p<\infty$.  In particular, the contribution of the terms involving the finite-rank operator $S$ are uncontrolled at the endpoints.  The lack of projections orthogonal to $v$ do not allow one to use the cancellation that was vital to our results in Theorem~\ref{thm:main}. Heuristically, the $(\log \lambda)^{-1}$ behavior near $\lambda=0$ does not provide enough smallness to improve the decay rate for large $x$ or $y$ by more than $(\log |x|)^{-1}$ or $(\log |y|)^{-1}$, which is not enough to reach the endpoints by considering only pointwise bounds.

\section{Integral estimates}\label{sec:ests}

Finally in this section, we collect the technical integral estimates necessary for the proofs in Sections~\ref{sec:swave}, \ref{sec:eval}, and \ref{sec:thmpf}.
We first note the following result on the pointwise decay of zero energy eigenfunctions, which follows from the proof of Lemma~5.5 in \cite{EG}.

\begin{lemma}\label{lem:efn decay}
	
	If $\psi$ is a zero energy eigenfunction, that is if $\psi \in L^2(\R^2)$ with $H\psi=0$, then  $\psi\in L^\infty(\R^2)$.  Furthermore, we have the pointwise bound
	$|\psi(x)|\les \la x\ra^{-1-}$.
	
\end{lemma}

We use the following simple integral bound,
\begin{lemma}\label{lem:bracket decay}
	
	If $0<\alpha, \beta$, $\alpha, \beta \neq n$ and $\alpha+\beta >n$ then
	$$
	\int_{\R^n} \la x-x_1\ra^{-\alpha}\la x_1 \ra^{-\beta} \, dx_1 \les \la x\ra^{-\min (\alpha, \beta, \alpha+\beta-n)}.
	$$
	
\end{lemma}

We now prove Lemma~\ref{lem:Lp kernel}.

\begin{proof}[Proof of Lemma~\ref{lem:Lp kernel}]
	
	We first note the integral kernel is admissible on the set  $|y|<2|x|$.  First, consider the subset when $|y|<\frac{1}{2}|x|$.  In this case $|x|-|y| \approx |x|$, so
	$$
		|K(x,y)|\les \frac{1}{\la x\ra^{3-\epsilon}}.
	$$
	Then, since $0<\epsilon<1$,
	\begin{multline*}
		\sup_{x} \int_{|y|<\frac{1}{2}|x|} |K(x,y)|\, dy+\sup_y \int_{|x|>2|y|} |K(x,y)|\, dx
		\les \sup_x \frac{1}{\la x\ra^{1-\epsilon}}+ \int_{\R^2}
		\frac{1}{\la x\ra^{3+\epsilon}}\, dx\les 1.
	\end{multline*}
	When $|x|\approx |y|$, if $|K(x,y)|\les \frac{1}{\la x\ra \la |x|-|y|\ra^2}$, after changing to polar co-ordinates, we have
	\begin{align*}
		\sup_{y} \int_{|x|\approx |y|} |K(x,y)| \, dx \les \sup_y \frac{1}{\la y\ra } \int_{|y|/2}^{2|y|} \frac{r}{\la r-|y|\ra^2}\, dr \les
		\sup_y \int_{\R} \frac{1}{\la r-|y|\ra^2}\, dr \les 1.
	\end{align*}
	By symmetry, an identical bound holds when $x$ and $y$ are reversed.  For the second pointwise bound, we have
	\begin{multline*}
	\sup_{y} \int_{|x|\approx |y|} |K(x,y)| \, dx \les \sup_y \frac{1}{\la y\ra^{2-\epsilon} } \int_{|y|/2}^{2|y|} \frac{r}{\la r-|y|\ra}\, dr\\ 
	\les
	\sup_y \frac{1}{\la y \ra^{1-\epsilon}} \int_{|y|/2}^{2|y|} \frac{1}{\la r-|y|\ra}\, dr
	\les \sup_y \frac{\la \log \la y\ra \ra}{\la y \ra^{1-\epsilon}}
	\les 1.
	\end{multline*}	
	Finally, we consider the region on which $|y|>2|x|$.  We first show that the operator is bounded when $p=1$.  In either case, we have that $|K(x,y)|\les \la x\ra^{-3+\epsilon}$.  Then, since this is in $L^1_x(\R^2)$ uniformly in $y$, $K$ is bounded on $L^1(\R^2)$.  This follows since if $|K(x,y)|\les k_1(x)k_2(y)$, then
	$$
		\|Ku\|_p=\bigg\| \int_{\R^2} K(x,y) u(y)\, dy \bigg\| \les \| k_1 \|_p \| k_2 \|_{p'} \|u\|_p.
	$$
	Thus $K$ is $L^p$-bounded provided one can control the norms of $k_1$ and $k_2$.
	
	Next, we show that $K$ is bounded on $L^p$ with $p$  arbitrarily large, but finite.  Equivalently, we take $p'>1$.  In either case, we may bound the kernel with $|K(x,y)|\les \la x\ra^{-1+\epsilon} \la y\ra^{-2}$.
	Since $\la x\ra^{-1+\epsilon} \in L^p(\R^2)$ for any $p>\frac{2}{1-\epsilon}$, we need only consider the integral for $y$.  Now, $\la y\ra^{-2} \in L^{p'}(\R^2)$ for any $p'>1$, or equivalently any $p<\infty$.  Under the assumptions, we have the $K$ is bounded on any $\frac{2}{1-\epsilon}<p<\infty$.  We note that the lower bound of $\frac{2}{1-\epsilon}$ is not sharp, however interpolation between these bounds suffices to prove the claimed range of $p$.

\end{proof}

We now prove lemmas \ref{lem:J''}, \ref{lem:J'}, and \ref{lem:J's}.

\begin{proof}[Proof of Lemma~\ref{lem:J''}]
We use the notation  $\omega(z)$ to denote any function satisfying
\be\label{omega}
|\omega^{(j)}(z)|\lesssim |z|^{-\frac12-j}\widetilde\chi(z),\,\,\,j=0,1,2,...,
\ee
and use the notation $\rho$ for functions supported on $[0,1]$ and satisfying
\be\label{rho}
|\rho^{(\ell)}(z)|\lesssim  1, \,\,\,\,\ell=0,1,2,... 
\ee
Recall that $J_0(z)$  and $H_0^-(z)=J_0(z)-iY_0(z)$ satisfy (see, e.g.,  \cite{EG})
$$
J_0^{\prime}(z)=z\rho(z)+e^{iz}\omega(z)+e^{-iz}\omega(z), 
$$
$$
J_0^{\prime\prime}(z)=\rho(z)+e^{iz}\omega(z)+e^{-iz}\omega(z), 
$$
$$
(H_0^-)^{\prime}(z)= \eta(z) +e^{-iz}\omega(z),
$$ 
where $\eta$ is supported on $[0,1]$ and  
\be\label{eta}
|\eta^{(\ell)}(z)|\lesssim z^{-1-\ell},\,\,\,\ell=0,1,2,...
\ee
Therefore the integral is
\begin{multline*}
\int_0^\infty \eta(\lambda r)\rho(\lambda s) \lambda^2 \chi(\lambda) d\lambda +\int_0^\infty   e^{i\lambda (r\pm s)} \omega(\lambda r) \omega(\lambda s)  \lambda^2 
\chi(\lambda)  d\lambda \\
+\int_0^\infty e^{\pm i\lambda s} \eta(\lambda r)  \omega(\lambda s) \lambda^2 
\chi(\lambda)  d\lambda + \int_0^\infty   e^{i\lambda r} \omega(\lambda r) \rho(\lambda s)  \lambda^2 
\chi(\lambda)  d\lambda\\=:A+B+C+D.
\end{multline*}
Using \eqref{rho} and \eqref{eta} we have
$$
A\les \int_0^{\min(1,r^{-1},s^{-1})} \lambda r^{-1} d\lambda \les \frac1{r \la r+s\ra^2}.
$$
Note that $C=0$ unless $s\gtrsim r $ and $s\gtrsim 1$, in which case  integrating  by parts twice using \eqref{eta} and \eqref{omega}, and noting that the effect of each derivative is division by $\lambda$, we obtain
$$
C\les \frac1{rs^{5/2}}\int_{s^{-1}}^1 \frac1{\lambda^{3/2}} d\lambda \les \frac1{rs^2} \les  \frac1{r \la r+s\ra^2}.
$$
Similarly, $D=0$ unless $r\gtrsim s$ and $r\gtrsim 1$. Three integration by parts give
$$
D\les \frac1{r^{7/2}}\int_{r^{-1}}^{1} \frac1{\lambda^{3/2}} d\lambda \les \frac1{r^3} \les  \frac1{r \la r+s\ra^2}.
$$
We consider the term $B$ only for the minus sign, the other case is easier. Note that $B=0$ unless $r,s\gtrsim 1$.  This term is easily seen to be bounded by $\frac1{\sqrt{rs}}$. In the case $|r-s| \ll 1$,  using \eqref{omega}, the integral is 
bounded by $\frac1{\sqrt{rs}}\les \frac1{\sqrt{rs}\la r-s\ra^{2-}}$. 
In the case $|r-s| \gtrsim 1$, integration by parts  yield
$$
B\les  \frac1{|r- s|} \Big|\int_0^\infty e^{i \lambda (r-s)} f(\lambda) d\lambda\Big|,$$
where
\be\label{fl}
f(\lambda)=\frac{\partial}{\partial \lambda}\big(\omega(\lambda r) \omega(\lambda s)  \lambda^2 
\chi(\lambda)\big)=O\big(\frac1{\sqrt{rs}}\big).  
\ee
The following bound is well known (for $L>1$)
\be\label{lipbound}
\Big|\int_0^\infty e^{\pm i \lambda L} f(\lambda) d\lambda\Big|\les 
\frac{\|f\|_{L^\infty}}{L}+ \int_0^\infty |f(\lambda+\frac{\pi}{L})-f(\lambda)| d\lambda.
\ee
Using the Mean Value Theorem in  \eqref{fl}, we have
$$
|f(\lambda+\frac{\pi}{L})-f(\lambda)|\les \frac{1}{L} \sup_{\rho\in(\lambda,\lambda+\pi/L)} \big| \frac{\partial^2}{\partial \rho^2}\big(\omega(\rho r) \omega(\rho s)  \rho^2 
\chi(\rho)\big|\les \frac{1}{L \lambda \sqrt{rs} }.
$$ 
Interpolating this with the bound in \eqref{fl}, we obtain
$$
|f(\lambda+\frac{\pi}{L})-f(\lambda)|\les \frac{1}{L^{1-} \lambda^{1-} \sqrt{rs}}
$$
Using this bound, \eqref{fl}, and \eqref{lipbound} for $B$, we have (for $L=|r-s|\gtrsim 1$)
$$B\les \frac1{|r- s|} \Big[\frac1{|r-s|\sqrt{rs}}+ \frac1{|r-s|^{1-} \sqrt{rs}}  \int_{0}^1 \frac{1}{\lambda^{1-}} d\lambda \Big]\les \frac{1}{\sqrt{rs}\la r-s\ra^{2-}}.
$$ 
\end{proof}
The proof of Lemma~\ref{lem:J's} is similar for the terms analogous to $A,C$, and $D$, for the $B$ term two integration by parts yield the bound instead of using \eqref{lipbound}.

\begin{proof}[Proof of Lemma~\ref{lem:J'}]
Using the notation of the proof of Lemma~\ref{lem:J''}, the integral is equal to
\begin{multline*}
sA+sD+   \int_0^\infty   e^{i\lambda (r\pm s)} \omega(\lambda r) \omega(\lambda s)  \lambda  
\chi(\lambda)  d\lambda + \int_0^\infty e^{\pm i\lambda s} \eta(\lambda r)  \omega(\lambda s) \lambda 
\chi(\lambda)  d\lambda \\=:sA+sD+B_1+C_1.
\end{multline*}
The bounds we obtained for $A, D$ above yield the required bound for these terms:
$$
sA+sD  \les \frac{s }{r \la r+s\ra^2}.
$$
The bound for $C_1$ is similar to the bound for $C$ above by integrating by parts only once.

Finally, in the cases $s\gg r$ or $r\gg s$, three integration by parts yield 
$$
B_1\les \frac{\min(r,s)^2}{\sqrt{rs}\max(r,s)^3}\les \frac{s }{r \la r+s\ra \la r - s\ra}.
$$
When $r\approx s$, an integration by parts yield
$$
B_1\les \frac{\log(r)}{ r \la r-s\ra}.
$$  
\end{proof}

\end{document}